\documentclass{amsart}
\usepackage{amsmath, amsthm, amssymb, amsfonts}
\usepackage{bm}
\usepackage{dsfont}
\usepackage[utf8]{inputenc}
\usepackage{rotate}
\usepackage{tikz}
\usepackage{tikz-cd}
\usepackage[arrow,matrix,curve]{xy} 	
\usepackage{xcolor}
\usepackage{lipsum}
\usepackage{caption}
\usepackage{subcaption}
\usepackage{todonotes}

\theoremstyle{plain}
\newtheorem{theorem}{Theorem}[section]
\newtheorem{proposition}[theorem]{Proposition}
\newtheorem{lemma}[theorem]{Lemma}

\newtheorem{conjecture}[theorem]{Conjecture}

\theoremstyle{definition}
\newtheorem{definition}[theorem]{Definition}

\newtheorem{example}[theorem]{Example}

\newcommand{\norm}[1]{\left\lVert#1\right\rVert}

\newcommand{\mathset}[2]{\left\{#1\middle\vert #2 \right\}}
\newcommand{\mathseq}[2]{\left(#1\right)_{#2}}

\newcommand{\card}[1]{{\# #1}}
\newcommand{\N}{\mathbb{N}}
\newcommand{\Z}{\mathbb{Z}}
\newcommand{\R}{\mathbb{R}}
\newcommand{\Q}{\mathbb{Q}}
\newcommand{\leb}{\mathbb{\lambda}}

\title{Weak Poissonian box correlations of higher order}
\date{\today}

\author{Jasmin Fiedler and Christian Weiß}
\address{{\bf{Ruhr West University of Applied Sciences,}}\\ {{Department of Natural Sciences, Duisburger Str. 100,}}\\{{45479 M\"ulheim an der Ruhr, Germany}}}
\email{jasmin.fiedler@hs-ruhrwest.de, christian.weiss@hs-ruhrwest.de}
\begin{document}

\maketitle
\begin{abstract}
    Poissonian pair correlations have sparked interest within the mathematical community, because of their number theoretic properties, and their connections to quantum physics and probability theory, in particular uniformly distributed random numbers. Rather recently, several generalizations of the concept have been introduced, including weak Poissonian pair correlations and $k$-th order Poissonian correlations. In this paper, we propose a new generalized concept called $(k,\beta)$-Poissonian box correlations. We study its properties and more specifically its relation to uniform distribution theory and discrepancy, random numbers and gap distributions.
\end{abstract}

\section{Introduction and main results}\label{sec:intro}

It is a practically relevant question how it is possible to find out if a given sequence $x_1, \ldots, x_n$ is (supposedly) a realization of a uniformly distributed random variable or a deterministically generated sequence, see e.g. the very well written article \cite{ES07} and references therein. While statistical tests are often applied in practice, there are also theoretical properties which a generic random sequence drawn independently from a (one-dimensional) uniform distribution satisfies.  Recall that the star-discrepancy of a finite sequence $x=(x_n)_{n=0}^N \in [0,1)$ with $N \in \mathbb{N}$ is defined by
$$D_N^*(x) := \sup_{[0,b) \subseteq [0,1)} \left| \frac{1}{N}\card{\mathset{x_i\in [0,b) }{ 1 \leq i \leq N }}
- \leb([0,b)) \right|,$$
where $\leb(\cdot)$ denotes Lebesgue measure. A famous classical property of independent uniformly distributed random variables is that the convergence of the star-discrepancy is of order $\mathcal{O}(N^{-1/2})$, see \cite{Nie92}. For (deterministic) sequences of real numbers it turns out often that their order of convergence for the discrepancy is determined by number theoretic properties of their underlying definition, e.g. van der Corput or Kronecker sequences, see \cite{KN74}. More recently, another generic property of random uniformly distributed sequences has attracted serious attention, see for example \cite{aistleitner:pair_corr_equidist}, \cite{grepstad:on_pair_corr_discr}, \cite{larcher:on_pair_corr_seq} and \cite{baker:equidist_seq_polynom}: a sequence $\mathseq{x_n}{n\in\N}$ is said to have Poissonian pair correlations if 
\begin{align} \label{eq:def:poissonian}
\lim_{N \to \infty} \frac{1}{N} \# \mathset{\substack{1\leq i,j\leq N\\i\neq j}}{\Vert x_i-x_j \Vert\leq \frac{s}{N}} = 2s
\end{align}
for all $s > 0$, where {$\norm{x} = \min \{x -\lfloor x \rfloor, 1- (x -\lfloor x \rfloor) \}$ represents the distance to the nearest integer. The concept of Poissonian pair correlations originated in quantum mechanics, but has been popularized in the mathematical community by Rudnick and Sarnak in \cite{rudnick:pair_corr_funct}.\\[12pt]
The property of a sequence $(x_n)_{n \in \mathbb{N}}$ to have Poissonian pair correlations is equivalent to the convergence
\begin{align} \label{eq:poissonian:functional}
    \lim_{N\to\infty}\frac{1}{N}\sum_{\substack{1\leq i,j\leq N\\i\neq j}}f(N((x_i-x_j)))=\int_{\R}f(x)\mathrm{d}x
\end{align}
for every $f\in C_C(\R)$, where $((x))$ is the signed distance to the nearest integer and $C_C(\R)$ stands for the set of all continuous real functions with compact support.\\[12pt]
Poissonian pair correlations are strongly connected to uniform distribution, which means for a sequence $\mathseq{x_n}{n\in\N}$ of numbers in the unit interval $[0,1)$ that
\[
\lim_{N\to\infty}\frac{1}{N}\card{\mathset{x_i\in [a,b)] }{ 1 \leq i \leq N }}=\leb([a,b))=b-a
\]
for all $0\leq a<b\leq 1$. The following theorem has been independently proven by Aistleitner, Lachmann and Pausinger in \cite{aistleitner:pair_corr_equidist} and by Grepstad and Larcher in \cite{grepstad:on_pair_corr_discr} and several more proofs have been found since then.
\begin{theorem} \label{thm:ALP:GL}
    Let $\mathseq{x_n}{n\in\N} \subset [0,1)$ be an arbitrary sequence. If the sequence has Poissonian pair correlations, then it is uniformly distributed.
\end{theorem}
Poissonian pair correlations have been generalized in many regards, including higher dimensions, see \cite{NP07}, \cite{hinrichs:mult_dim_poiss_pair}, \cite{bera:on_high_dim_poiss_pair_corr}, \cite{Ste20}, the p-adic integers, see \cite{Wei23}, weak versions, see \cite{NP07} and \cite{hauke:weak_poiss_corr},  and to higher orders, see \cite{cohen:poiss_pair_higher_diff}, \cite{hauke:poiss_corr_higher_order}. In the ultimately mentioned paper, so-called $k$-th order Poissonian correlations have been introduced. According to the definition therein, a sequence $\mathseq{x_n}{n\in\N}$ in $[0,1)$ has $k$-th order Poissonian correlations for $k\geq 2$ if 
    \begin{equation}\label{eq:def_poiss_k_corr}
       \lim_{N\to\infty}\frac{1}{N}\sum_{\substack{1\leq i_1,\ldots,i_k\leq N\\\text{distinct}}}f(N((x_{i_1}-x_{i_k})),\ldots,N((x_{i_{k-1}}-x_{i_k})))=\int_{\R^{k-1}}f(x)\mathrm{d}x 
    \end{equation}
for every $f\in C_C(\R^{k-1})$. Thus, the generalization of Hauke and Zafeiropoulos is based on the functional property \eqref{eq:poissonian:functional}.
In \cite{hauke:poiss_corr_higher_order}, the authors showed the following generalization of Theorem~\ref{thm:ALP:GL}.
\begin{theorem}[{\cite[Theorem 1.1]{hauke:poiss_corr_higher_order}}]
    If a sequence $\mathseq{x_n}{n\in\N}$ of numbers in $[0,1)$ has Poissonian $k$–correlations for some $k\geq 2$, then it is uniformly distributed.
\end{theorem}
Also the opposite is generically true. This is relatively straightforward to show if $k=2$, see \cite{Mar07}, while more sophisticated techniques are required for higher orders.
\begin{theorem}[{\cite[Appendix B]{hauke:poiss_corr_higher_order}}]
    Let $\mathseq{X_n}{n\in\N}$ be a sequence of independent random variables, drawn from uniform distribution on $[0,1)$. Then the sequence almost surely has Poissonian $k$–correlations for all $k\geq 2$.
\end{theorem}
In light of the original definition of Poissonian pair correlations, one might ask whether the approach from \cite{hauke:poiss_corr_higher_order} is equivalent to the equally natural generalization
    \begin{equation}\label{eq:def_k_beta_poiss_box_corr}
      \lim_{N\to\infty}\frac{1}{N}\# \mathset{\substack{1\leq i_1,\ldots,i_k\leq N\\\text{distinct}}}{\Vert x_{i_1}-x_{i_k} \Vert\leq \frac{s_1}{N},\ldots,\Vert x_{i_{k-1}}-x_{i_k} \Vert\leq \frac{s_{k-1}}{N}} = \prod_{l=1}^{k-1}2s_l  
    \end{equation}
for all $s_1,\ldots,s_{k-1}>0$, which is closer to~\eqref{eq:def:poissonian}. It is easy to show that $k$-th order Poissonian correlations in the sense of Hauke and Zafeiropoulos imply \eqref{eq:def_k_beta_poiss_box_corr} (the proof is essentially that of Proposition~\ref{prop:equiv_def_poiss_corr}), but it is not clear whether or not the two notions are equivalent. In \cite[Appendix A]{hauke:poiss_corr_higher_order} an heuristic argument is given, why the notions might differ. The authors state that their symmetry argument for the case $k=2$ does not suffice to show equivalence of the two notions in general and stress the difficulty of the question by hat. Also, we were not able to find a rigorous proof for this claim and therefore leave it here as an open question. No matter if the definitions differ or not, we think that both possible equally natural generalizations deserve further investigation. In order to distinguish the definition based on \eqref{eq:def_k_beta_poiss_box_corr} from the one stemming from \cite{hauke:poiss_corr_higher_order}, we will speak of $k$-th order Poissonian box correlations for sequences satisfying \eqref{eq:def_k_beta_poiss_box_corr}, and use the same linguistic principal for similar definitions to follow.\\[12pt]
More specifically, we introduce $(k,\beta)$-Poissonian box correlations in this paper.
Let $(x_n)_{n \in \mathbb{N}}$ be a sequence in $[0,1)$. For $\beta\in(0,1]$, $k\in\mathbb{N}$ and $s_1,\ldots,s_{k-1}>0$ define
\[
R_{k,\beta}(s_1,\ldots,s_{k-1},N)=\frac{\card{\mathset{\substack{1\leq i_1,\ldots,i_k\leq N\\\text{distinct}}}{\Vert x_{i_1}-x_{i_k} \Vert\leq \frac{s_1}{N^\beta},\ldots,\Vert x_{i_{k-1}}-x_{i_k} \Vert\leq \frac{s_{k-1}}{N^\beta}}}}{N^{k-(k-1)\beta}}.
\]
We say that a sequence $\mathseq{x_n}{n\in\N}$ has $(k,\beta)$-Poissonian box correlations if
\[
\lim_{N\to \infty}R_{k,\beta}(s_1,\ldots,s_{k-1},N)=\prod_{l=1}^{k-1}2s_l.
\]
for all $s_1,\ldots,s_{k-1}>0$. Moreover, $(k,\beta)$-Poissonian correlations are defined analogously to \eqref{eq:def_poiss_k_corr}, meaning that
\[
\lim_{N\to\infty}R_{k,\beta}(f,N)=\int_{\R^{k-1}}f(x)\mathrm{d}x 
\]
for every $f\in C_C(\R^{k-1})$ where
\[
R_{k,\beta}(f,N):=\frac{1}{N^{k-(k-1)\beta}}\sum_{\substack{1\leq i_1,\ldots,i_k\leq N\\\text{distinct}}}f(N^\beta((x_{i_1}-x_{i_k})),\ldots,N^\beta((x_{i_{k-1}}-x_{i_k}))).
\]
In Section \ref{sec:kmb_poiss_corr} we give an alternative characterization of $(k,\beta)$-Poissonian box correlations which relies on an integral over the functions
\begin{align}
            G_\beta(s_1,\ldots,s_{k-1},&t_1,\ldots,t_{k-1},N):=\nonumber\\&\frac{\card{\mathset{\substack{1\leq i_1,\ldots,i_{k-1}\leq N\\\mathrm{distinct}}}{\Vert x_{i_1}-t_1\Vert\leq \frac{s_1}{2N^\beta},\ldots,\Vert x_{i_{k-1}}-t_{k-1}\Vert\leq \frac{s_{k-1}}{2N^\beta}}}}{N^{k-1-(k-1)\beta}}.\label{eq:def_G_beta}
\end{align}
        and
\begin{align}
            H_\beta(s_1,\ldots,s_{k-1},&t_1,\ldots,t_{k-1},N):=\nonumber\\&\frac{\card{\mathset{1\leq i\leq N}{\Vert x_i-t_1\Vert\leq \frac{s_1}{2N^\beta},\ldots,\Vert x_i-t_{k-1}\Vert\leq \frac{s_{k-1}}{2N^\beta}}}}{N^{1-(k-1)\beta}}\label{eq:def_H_beta}
\end{align}
where $s_1,\ldots, s_{k-1}>0$ and $t_1,\ldots,t_{k-1}\in\R$. 
Note that the functions $G_\beta$ and $H_\beta$ agree for $k=2$ but differ for $k>3$. This implies among others that $(k,\beta)$-Poissonian box correlations are significantly different in the case $k \geq 3$ than for $k=2$ because the combinatorial behavior of the points in the sequence contribute in two different ways to the property of having $(k,\beta)$-Poissonian box correlations which is captured by the function $G_\beta$ and $H_\beta$. This is another reason, why we think that $(k,\beta)$-Poissonian box correlations are worth to be discovered. More precisely, the following holds.
\begin{theorem}\label{thm:poiss_corr_eq_def_fun_G_H}
    Let $\mathseq{x_n}{n\in\N}$ be a sequence in $[0,1)$. Let $\beta\in(0,1]$ and $k\geq 2$. Then the following are equivalent:
    \begin{enumerate}
        \item[(i)] The sequence has $(k,\beta)$-Poissonian box correlations.
        \item[(ii)] For any $s_1,\ldots,s_{k-1}>0$ we have
        \begin{align*}
            \lim_{N\to\infty}\int_0^1\ldots\int_0^1 & G_\beta\cdot H_\beta\mathrm{d}t_{k-1}\ldots\mathrm{d}t_{1}=\begin{cases}
                \prod_{i=1}^{k-1}s_i^2&\text{if }\beta<1\\
                \varphi(s_1,\ldots,s_{k-1})&\text{if }\beta=1
            \end{cases},
        \end{align*}
        where
        \[
        \varphi(s_1,\ldots,s_{k-1})=\begin{cases}
                1&\text{if }k=1\\
                \prod_{i=1}^{k-1}s_i^2+\sum_{i=1}^{k-1}s_i\varphi(s_1,\ldots,\hat{s_j},\ldots,s_{k-1})&\text{if }k\geq 2
            \end{cases}
        \]
        and the hat operator means that the respective element is omitted.
    \end{enumerate}

\end{theorem}
Theorem~\ref{thm:poiss_corr_eq_def_fun_G_H} is a generalization of \cite[Lemma 9]{hauke:weak_poiss_corr}, where the claim has been shown for $k=2$. More generally, we notice that since $k$ is not necessarily even, it is not possible in the general setting to use the square of a single function on the left-hand side of the integral. We also believe that Theorem~\ref{thm:poiss_corr_eq_def_fun_G_H} is an important stepping stone in proving that stronger correlations imply weaker ones, see Conjecture \ref{conj:poiss_k_stronger_impl_weaker} later in this paper.\\[12pt]
While Poissonian pair correlations are a generic property of uniformly distributed sequences, it turns out hard to find explicit examples in the case $\beta = 1$. Some of the few known examples can be found in \cite{BMV15,LST21,LT22}. One obstacle in finding such sequences lies in the gap structure of a sequence: for an ordered set $
x_1 \leq x_2 \leq \ldots \leq x_N \subset [0,1)$ of points in $[0,1)$, the gap lengths are defined by $\Vert x_{i+1}-x_{i}\Vert$, $1\leq i<N$, where we set $x_{N+1} = x_1$. For an arbitrary set $\{x_1,\ldots,x_N\}$ of points in $[0,1)$, a gap is defined as a gap in the corresponding ordered set. Indeed, the number of different gap lengths has a great influence on the Poissonian pair correlations property of sequences. In \cite[Proposition 1]{larcher:som_neg_results_poiss_pair_corr}, Larcher and Stockinger prove the following.
\begin{theorem}
    Let $\mathseq{x_n}{n\in\N}$ be a sequence in $[0,1)$. If there is some $Z\in\N$, some $\gamma>0$ and infinitely many $N\in\N$ such that $x_1,\ldots,x_N$ has a subset of size $M\geq \gamma N$ which has at most $Z$ different distinct distances between neighboring elements, then $\mathseq{x_n}{n\in\N}$ does not have Poissonian pair correlations.
\end{theorem}
Indeed, the same also holds true for $(k,1)$-Poissonian box correlations.
\begin{theorem}\label{thm:gap_strucure_poiss_k_corr}
    Let $\mathseq{x_n}{n\in\N}$ be a sequence in $[0,1)$. If there is some $Z\in\N$, some $\gamma>0$ and infinitely many $N\in\N$ such that $x_1,\ldots,x_N$ has at most $Z$ different distinct distances between neighboring elements, then $\mathseq{x_n}{n\in\N}$ does not have $(k,1)$-Poissonian box correlations for any $k\geq 2$. 
\end{theorem}
Despite the fact that our result is more general than the one in \cite{larcher:som_neg_results_poiss_pair_corr}, we are able to give a significantly more compact proof. However for $\beta < 1$, examples can be constructed by the following theorem which generalizes results from \cite{weiss:some_conn_discr_fin_gap_pair_corr}.
\begin{theorem}\label{thm:poiss_k_corr_and_discrepancy}Let $x=\mathseq{x_n}{n\in\N}$ be a sequence in $[0,1]$ with $D_N(x)=o(N^{-(1-\varepsilon)})$ for $0<\varepsilon<1$. Then the sequence has $(k,\beta)$-Poissonian box correlations for all $0 < \beta \leq 1 - \varepsilon$.
\end{theorem}
In combination, Theorems \ref{thm:poiss_k_corr_and_discrepancy} and \ref{thm:gap_strucure_poiss_k_corr} provide interesting and non-trivial insight into some highly relevant examples of uniformly distributed sequences: both, Kronecker sequences for badly approximable $\alpha \in \mathbb{R}$ as well as Van-der-Corput sequences in arbitrary base, satisfy $D_N(x)=o(N^{-(1-\varepsilon)})$ for every $0<\varepsilon<1$ (see e.g. \cite{KN74}) and  hence Theorem \ref{thm:poiss_k_corr_and_discrepancy} implies that they have $(k,\beta)$-Poissonian box correlations for all $\beta<1$. At the same time Theorem \ref{thm:gap_strucure_poiss_k_corr} states that they do not have $(k,1)$-Poissonian box correlations.\\[12pt]
The remainder of the paper is organized as follows: In Section~\ref{sec:kmb_poiss_corr} we will show some basic Properties of $(k,\beta)$-box correlated sequences. Afterwards, we will in Section \ref{sec:gap_structure} take a look at the relation between $(k,1)$-box Poissonian correlations and the gap structures of sequences leading to the proof of Theorem~\ref{thm:gap_strucure_poiss_k_corr}. Furthermore, we will discuss some explicit examples. Finally, we discuss in Section~\ref{sec:future} open questions related to the content of this paper.\\[12pt]
Finally, we want to give a formal definition of the notation which we will use throughout the article: given a real number $x$, we write 
\[
\{x\}=x-\lfloor x\rfloor
\]
for the fractional part of $x$. 
The symbol
\[
((x))= \begin{cases}
\{x\}&\text{if }0\leq \{x\}\leq\frac12\\
\{x\}-1&\text{if }\{x\}\geq\frac12
\end{cases}
\]
denotes the signed distance of $x$ to the nearest integer and
\[
\Vert x \Vert=\min\mathset{\vert x-z\vert}{z\in\Z}=\vert((x))\vert
\]
is the distance of $x$ to the nearest integer. Given some $n\in\N$ we write $[n]$ for the set $\{1,\ldots,n\}$. For $x\in\R$, let $\{x\}^+=\max\{0,x\}$ describe the positive part and for a set $\{x_1,\ldots,x_n\}$ and $1\leq k\leq n$ we write 
\[
\{x_1,\ldots,\hat{x_k},\ldots,x_n\}=\{x_1,\ldots,x_n\}\setminus\{x_k\}
\]
for the set of elements without $x_k$. We will use the same notation for tuples and vectors. The closed ball of radius $r\geq 0$ around $x\in [0,1)$ is defined as
\[
B(x,r)=\mathset{y\in X}{\Vert x-y\Vert\leq r}.
\]
Finally, $\card{S}$ denotes the cardinality of a set $S$ and $C_C(\mathbb{R}^{k-1})$ is the set of continuous functions on $\mathbb{R}^{k-1}$ with compact support.

\paragraph{Acknowledgement.} The authors would like to thank Manuel Hauke for his comments on an earlier version of this paper.

\section{$(k,\beta)$-Poissonian box correlations}\label{sec:kmb_poiss_corr}

In this section, the concept of $(k,\beta)$-Poissonian box correlations is investigated more deeply. In order to do so, we will at first discuss its connection to $(k,\beta)$-Poissonian correlations. For this discussion the following theorem which is an analogue of \cite[Proposition A]{hauke:poiss_corr_higher_order} turns out to be useful.
\begin{proposition}\label{prop:equiv_def_poiss_corr}
Let $x=\mathseq{x_n}{n\in\N}$ be a sequence in $[0,1]$ and $k\in\mathbb{N}$ with $2\leq k$. The following are equivalent:
\begin{enumerate}
\item[(i)] The sequence $x$ has $(k,\beta)$-Poissonian correlations.
\item[(ii)] For all $f\in C_C(\R^{k-1})$ we have
\[
\lim_{N\to \infty}R_{k,\beta}^\prime(f,N)=\int_{\R^{k-1}}f(x)\mathrm{d}x
\]
where
\[
R_{k,\beta}^\prime(f,N)=\frac{1}{N^{k-(k-1)\beta}}\sum_{\substack{1\leq i_1,\ldots,i_k\leq N\\\text{distinct}}}f(N^\beta((x_{i_1}-x_{i_2})),N^\beta((x_{i_2}-x_{i_3})),\ldots,N^\beta((x_{i_{k-1}}-x_{i_k}))).
\]
\item[(iii)] For every rectangle $B=[a_1,b_1]\times[a_2,b_2]\times\ldots\times[a_{k-1},b_{k-1}]$ we have
\[
\lim_{N\to\infty}R_{k,\beta}(\mathds{1}_B,N)=\leb(B).
\]
\end{enumerate}
\end{proposition}
\begin{proof}
First, we show $(i)\implies(iii)$. For $\delta>0$ define 
\[
B_\delta^+=[a_1-\delta,b_1+\delta]\times[a_2-\delta,b_2+\delta]\times\ldots\times[a_{k-1}-\delta,b_{k-1}+\delta]
\]
\[
B_\delta^-=[a_1+\delta,b_1-\delta]\times[a_2+\delta,b_2-\delta]\times\ldots\times[a_{k-1}+\delta,b_{k-1}-\delta]
\]
For arbitrary $\varepsilon>0$, let $\delta>0$ such that $\leb(B_\delta^+)\leq\leb(B)+\varepsilon$ and $\leb(B_\delta^-)\geq\leb(B)-\varepsilon$. For $1\leq i\leq k-1$ choose a continuous function $f^+_i:\R\to [0,1] $ with $f^+_i|_{[a_i,b_i]}\equiv 1$, and $f_i^+(x)=0$ for $x\leq a_i-\delta$ and $f_i^+(x)=0$ for $x\geq b_i+\delta$. Similarly choose a continuous function $f^-_i:\R\to [0,1] $ with $f^-_i|_{[a_i+\delta,b_i-\delta]}\equiv 1$, and $f_i^-(x)=0$ for $x\leq a_i$ and $f_i^+(x)=0$ for $x\geq b_i$. Further define $f^+,f^-\in C_C(\R^{k-1})$ as $f^+(x_1,x_2,\ldots,x_{k-1})=\prod_{i=1}^{k-1}f^+_i(x_i)$ and $f^-(x_1,x_2,\ldots,x_{k-1})=\prod_{i=1}^{k-1}f^-_i(x_i)$. Then 
\[
R_{k,\beta}(f^-,N)\leq R_{k,\beta}(\mathds{1}_B,N)\leq R_{k,\beta}(f^+,N)
\]
and thereby
\begin{align*}
\leb(B)-\varepsilon&\leq\leb(B_\delta^-)=\int_{\R^{k-1}}\mathds{1}_{B_\delta^-}(x)\mathrm{d}x\leq\int_{\R^{k-1}}f^-(x)\mathrm{d}x\\
&=\lim_{N\to\infty}R_{k,\beta}(f ^-,N)\leq \lim_{N\to\infty}R_{k,\beta}(\mathds{1}_B,N)\leq \lim_{N\to\infty}R_{k,\beta}(f^+,N)\\
&\leq\int_{\R^{k-1}}\mathds{1}_{B_\delta^+}(x)\mathrm{d}x\leq \leb(B_\delta^+)=\leb(B)+\varepsilon.
\end{align*}
Since this is true for all $\varepsilon>0$, the desired result is derived.\\
$(i)\implies(iii)$ can be shown using measure-theoretic induction.\\
$(i)\Leftrightarrow(ii)$ follows verbatim as in the proof of \cite[Proposition A]{hauke:poiss_corr_higher_order}.
\end{proof}
In order to keep the presentation compact, we use the abbreviation
$$R_{k,\beta}(s,N)=R_{k,\beta}(\underbrace{s,\ldots,s}_{k-1\text{ times}},N)=R_{k,\beta}\left(\mathds{1}_{[-s,s]^{k-1}},N\right).$$ 
Recall from the introduction that a sequence has $(k,\beta)$-Poissonian box correlations if
\[
\lim_{N\to\infty}R_{k,\beta}(s_1,\ldots,s_{k-1},N)=\prod_{l=1}^{k-1}2s_l
\]
for all $s_1,\ldots,s_{k-1}>0$.\\[12pt]
Obviously, the definition of $(k,\beta)$-Poissonian box correlations is a special case of Proposition~\ref{prop:equiv_def_poiss_corr}~(iii). This means that every sequence with $(k,\beta)$-Poissonian correlations also has $(k,\beta)$-Poissonian box correlations. For $k=2$, a simple symmetry argument shows that the converse is also true, see \cite[Appendix A]{hauke:poiss_corr_higher_order}. For $k>2$, however, this argument fails. Although this does not necessarily imply that the two definitions are indeed distinct, for the time being it is reasonable to use two different names. 

We now move towards the proof of Theorem~\ref{thm:poiss_corr_eq_def_fun_G_H}. First, we show the following equivalent characterization of $(k,\beta)$-Poissonian box correlation.

\begin{proposition}\label{prop:poiss_corr_eq_def_fun_G_H}
    Let $\mathseq{x_n}{n\in\N}$ be a sequence in $[0,1)$. Let $\beta\in(0,1]$ and $k\geq 2$. Then the following are equivalent:
    \begin{enumerate}
        \item[(i)] For any $s_1,\ldots,s_{k-1}>0$ we have
        \[
            \lim_{N\to\infty}R_{k,\beta}(s_1,\ldots,s_{k-1},N)=\prod_{i=1}^{k-1}2s_i.
        \]
        \item[(ii)] For any $s_1,\ldots,s_{k-1}>0$ we have
        \[
            \lim_{N\to\infty}\int_0^{s_1}\ldots\int_0^{s_{k-1}} R_{k,\beta}(\sigma_1,\ldots,\sigma_{k-1},N)\mathrm{d}\sigma_{k-1}\ldots\mathrm{d}\sigma_{1}=\prod_{i=1}^{k-1}s_i^2.
        \]
    \end{enumerate}

\end{proposition}

One key ingredient for the proof of Proposition~\ref{prop:poiss_corr_eq_def_fun_G_H} is the following assertion. 
\begin{lemma}\label{lem:in_and_out_formula}
    Let $k\in\N$ and let $0\leq a_i<b_i$ be real numbers for $1\leq i\leq k$. Then
    \[
        \mathds{1}_{\prod_{i=1}^{k}[a_i,b_i]}=\sum_{P\subseteq[k]}(-1)^{\card{P}}\mathds{1}_{\prod_{i\in P}[0,a_i]}\cdot\mathds{1}_{\prod_{i\notin P}[0,b_i]}
    \]
    holds almost everywhere with respect to the Lebesgue measure, 
    where the sum is taken over all subsets $P$.
\end{lemma}
While the result is surely well-known to experts, we were not able to find a reference in the literature and hence include a proof in this paper for the sake of completeness.
\begin{proof}
    We prove this via an induction on $k$. The assertion is obvious for $k=1$. Now let $k\geq 2$ and assume that the claim holds for $k-1$. Then
    \begin{align*}
        \mathds{1}_{\prod_{i=1}^{k}[a_i,b_i]}=&\mathds{1}_{\prod_{i=1}^{k-1}[a_i,b_i]}\cdot\mathds{1}_{[0,b_i]}-\mathds{1}_{\prod_{i=1}^{k-1}[a_i,b_i]}\cdot\mathds{1}_{[0,a_i]}\\
        =&\sum_{P\subseteq[k-1]}(-1)^{\card{P}}\mathds{1}_{\prod_{i\in P}[0,a_i]}\cdot\mathds{1}_{\prod_{i\notin P}[0,b_i]}\cdot\mathds{1}_{[0,b_i]}\\
        &-\sum_{P\subseteq[k-1]}(-1)^{\card{P}}\mathds{1}_{\prod_{i\in P}[0,a_i]}\cdot\mathds{1}_{\prod_{i\notin P}[0,b_i]}\cdot\mathds{1}_{[0,a_i]}\\
        =&\sum_{P\subseteq[k-1]}(-1)^{\card{P}}\mathds{1}_{\prod_{i\in P}[0,a_i]}\cdot\mathds{1}_{\prod_{i\notin P}[0,b_i]}\cdot\mathds{1}_{[0,b_i]}\\
        &+\sum_{P\subseteq[k-1]}(-1)^{\card{P}+1}\mathds{1}_{\prod_{i\in P}[0,a_i]}\cdot\mathds{1}_{\prod_{i\notin P}[0,b_i]}\cdot\mathds{1}_{[0,a_i]}\\
        &=\sum_{P\subseteq[k]}(-1)^{\card{P}}\mathds{1}_{\prod_{i\in P}[0,a_i]}
    \end{align*}
almost everywhere.
\end{proof}
This sets us into the position to show Proposition~\ref{prop:poiss_corr_eq_def_fun_G_H}.
\begin{proof}[Proof of Proposition~\ref{prop:poiss_corr_eq_def_fun_G_H}]
    We start by showing that (i) implies (ii). Let $M\geq 1$ be an arbitrary integer and for $1\leq i\leq k-1$ let $0\leq j_i\leq M-1$ be arbitrary as well. First, by the monotonicity of the $R_{k,\beta}$ function in its first $k-1$ parameters we have
    \begin{align*}
    \frac{\prod_{i=1}^{k-1}s_i}{M^{k-1}}&R_{k,\beta}\left(\frac{j_1s_1}{M},\ldots,\frac{j_{k-1}s_{k-1}}{M},N\right)\\
    &\leq  \int_{\frac{j_1s_1}{M}}^{\frac{(j_1+1)s_1}{M}}\ldots\int_{\frac{j_{k-1}s_{k-1}}{M}}^{\frac{(j_{k-1}+1)s_{k-1}}{M}}R_{k,\beta}(\sigma_1,\ldots,\sigma_{k-1},N)\mathrm{d}\sigma_{k-1}\ldots\mathrm{d}\sigma_{1}\\
    &\leq\frac{\prod_{i=1}^{k-1}s_i}{M^{k-1}}R_{k,\beta}\left(\frac{(j_1+1)s_1}{M},\ldots,\frac{(j_{k-1}+1)s_{k-1}}{M},N\right)
    \end{align*}
    and therefore 
    \begin{align} \label{eq:lem_eq_poiss_k_corr_ineq_int}
    \begin{split}
    & \frac{\prod_{i=1}^{k-1}s_i}{M^{k-1}}\sum_{1\leq j_1,\ldots ,j_{k-1}\leq M-1}R_{k,\beta}\left(\frac{j_1s_1}{M},\ldots,\frac{j_{k-1}s_{k-1}}{M},N\right)\\
    &\leq \int_{0}^{s_1}\ldots\int_{0}^{s_{k-1}}R_{k,\beta}(\sigma_1,\ldots,\sigma_{k-1},N)\mathrm{d}\sigma_{k-1}\ldots\mathrm{d}\sigma_{1}\\
    &\leq\frac{\prod_{i=1}^{k-1}s_i}{M^{k-1}}\sum_{1\leq j_1,\ldots ,j_{k-1}\leq M-1}R_{k,\beta}\left(\frac{(j_1+1)s_1}{M},\ldots,\frac{(j_{k-1}+1)s_{k-1}}{M},N\right).
    \end{split}
    \end{align}
    By taking the limit $N\to\infty$ and using (i) afterwards we get
    \begin{align*}
        \lim_{N\to\infty}\sum_{1\leq j_1,\ldots ,j_{k-1}\leq M-1}R_{k,\beta}&\left(\frac{j_1s_1}{M},\ldots,\frac{j_{k-1}s_{k-1}}{M},N\right)=\sum_{1\leq j_1,\ldots ,j_{k-1}\leq M-1}\prod_{i=1}^{k-1}\frac{2j_is_i}{M}\\
        &=\frac{2^{k-1}\prod_{i=1}^{k-1}s_i}{M^{k-1}}\sum_{j_1=1}^{M-1}j_1\cdot\ldots\cdot\sum_{j_{k-1}=1}^{M-1}j_{k-1}\\
        &=\frac{2^{k-1}\prod_{i=1}^{k-1}s_i}{M^{k-1}}\frac12(M-1)M\cdot\ldots\cdot\frac12(M-1)M\\
        &=(M-1)^{k-1}\prod_{i=1}^{k-1}s_i
    \end{align*}
    and in the same way 
    \[
    \lim_{N\to\infty}\sum_{1\leq j_1,\ldots ,j_{k-1}\leq M}R_{k,\beta}\left(\frac{(j_1+1)s_1}{M},\ldots,\frac{(j_{k-1}+1)s_{k-1}}{M},N\right)=(M+1)^{k-1}\prod_{i=1}^{k-1}s_i.
    \]
    Plugging these two equations into \eqref{eq:lem_eq_poiss_k_corr_ineq_int} yields
    \begin{align*}
    \frac{\prod_{i=1}^{k-1}s_i}{M^{k-1}}(M-1)^{k-1}\prod_{i=1}^{k-1}s_i&\leq \int_{0}^{s_1}\ldots\int_{0}^{s_{k-1}}R_{k,\beta}(\sigma_1,\ldots,\sigma_{k-1},N)\mathrm{d}\sigma_{k-1}\ldots\mathrm{d}\sigma_{1}\\
    &\leq\frac{\prod_{i=1}^{k-1}s_i}{M^{k-1}}(M+1)^{k-1}\prod_{i=1}^{k-1}s_i.
    \end{align*}
    Letting $M\to\infty$ proves (ii).\\[12pt]
    Next we show that (ii) implies (i). Let $\varepsilon>0$. First, we want to express the integral 
    \[
    I^-(N,\varepsilon):=\int_{s_1-\varepsilon}^{s_1}\ldots\int_{s_{k-1}-\varepsilon}^{s_{k-1}}R_{k,\beta}(\sigma_1,\ldots,\sigma_{k-1},N)\mathrm{d}\sigma_{k-1}\ldots\mathrm{d}\sigma_{1}
    \]
    in terms of integrals starting at $0$. In other words we want to express the box $\prod_{i=1}^{k-1}[s_i-\varepsilon,s_i]$ in terms of boxes starting at zero. Lemma \ref{lem:in_and_out_formula} leads to
    \[
        \mathds{1}_{\prod_{i=1}^{k-1}[s_i-\varepsilon,s_i]}=\sum_{P\subseteq[k-1]}(-1)^{\card{P}}\mathds{1}_{\prod_{i\in P}[0,s_i-\varepsilon]}\cdot\mathds{1}_{\prod_{i\notin P}[0,s_i]}
    \]
    almost everywhere. Together with assumption (ii) this yields
    \begin{equation}\label{eq:lem_eq_poiss_k_corr_eq_lim_int}
        \lim_{N\to\infty}I^-(N,\varepsilon)=\sum_{P\subseteq[k-1]}(-1)^{\card{P}}\prod_{i\in P}(s_i-\varepsilon)^2\prod_{i\notin P}s_i^2.
    \end{equation}
    For any $k\geq 2$ thus
    \begin{equation}\label{eq:lem_eq_poiss_k_corr_eq_si2}
        \sum_{P\subseteq[k-1]}(-1)^{\card{P}}\prod_{i\in P}(s_i-\varepsilon)^2\prod_{i\notin P}s_i^2=\prod_{i=1}^{k-1}2s_i\varepsilon+\mathcal{O}(\varepsilon^k).
    \end{equation}
    Hence,
    \[
    \lim_{\varepsilon\to 0}\lim_{N\to\infty}\frac{1}{\varepsilon^{k-1}}I^-(N,\varepsilon)=\prod_{i=1}^{k-1}2s_i.
    \]
    Similarly it can be shown for
    \[
    I^+(N,\varepsilon):=\int_{s_1}^{s_1+\varepsilon}\ldots\int_{s_{k-1}}^{s_{k-1}+\varepsilon}R_{k,\beta}(\sigma_1,\ldots,\sigma_{k-1},N)\mathrm{d}\sigma_{k-1}\ldots\mathrm{d}\sigma_{1}
    \]
    that
    \[
    \lim_{\varepsilon\to 0}\lim_{N\to\infty}\frac{1}{\varepsilon^{k-1}}I^+(N,\varepsilon)=\prod_{i=1}^{k-1}2s_i.
    \]
    The monotony of $R_{k,\beta}$ implies
    \[
    \frac{1}{\varepsilon^{k-1}}I^-(N,\varepsilon)\leq R_{k,\beta}(s_1,\ldots,s_{k-1},N)\leq \frac{1}{\varepsilon^{k-1}}I^+(N,\varepsilon),
    \]
    and thus we have shown that $\lim_{N\to\infty}R_{k,\beta}(s_1,\ldots,s_{k-1},N)=\prod_{i=1}^{k-1}2s_i$, which is assertion (i).

    \end{proof}
    We now have all the ingredients necessary for proving Theorem \ref{thm:poiss_corr_eq_def_fun_G_H}.
    \begin{proof}[Proof of Theorem~\ref{thm:poiss_corr_eq_def_fun_G_H}]
     We can rewrite the functions \eqref{eq:def_G_beta} and \eqref{eq:def_H_beta} as
    \[
    G_\beta(s_1,\ldots,s_{k-1},t_1,\ldots,t_{k-1},N)=\frac{1}{N^{k-1-(k-1)\beta}}\sum_{\substack{1\leq i_1,\ldots,i_{k-1}\leq N\\\mathrm{distinct}}}\prod_{l=1}^{k-1}\mathds{1}_{B\left(x_{i_l},\frac{s_l}{2N^\beta}\right)}(t_l)
    \]
    and     
    \[
    H_\beta(s_1,\ldots,s_{k-1},t_1,\ldots,t_{k-1},N)=\frac{1}{N^{1-(k-1)\beta}}\sum_{i=1}^N\prod_{l=1}^{k-1}\mathds{1}_{B\left(x_i,\frac{s_l}{2N^\beta}\right)}(t_l).
    \]
    Therefore,
    \begin{align*}
    G_\beta&(s_1,\ldots,s_{k-1},t_1,\ldots,t_{k-1},N)\cdot H_\beta(s_1,\ldots,s_{k-1},t_1,\ldots,t_{k-1},N)\\
    &=\frac{1}{N^{k-2(k-1)\beta}}\sum_{\substack{i_1,\ldots,i_{k}\leq N\\\mathrm{distinct}}}\prod_{l=1}^{k-1}\mathds{1}_{B\left(x_{i_l},\frac{s_l}{2N^\beta}\right)}(t_l)\mathds{1}_{B\left(x_{i_k},\frac{s_l}{2N^\beta}\right)}(t_l)\\&+\frac{1}{N^{k-2(k-1)\beta}}\sum_{\substack{i_1,\ldots,i_{k-1}\leq N\\\mathrm{distinct}}}\sum_{j=1}^{k-1}\prod_{l=1}^{k-1}\mathds{1}_{B\left(x_{i_j},\frac{s_l}{2N^\beta}\right)}(t_l)\mathds{1}_{B\left(x_{i_l},\frac{s_l}{2N^\beta}\right)}(t_l).
    \end{align*}
    It follows that
    \begin{align*}
        \int_0^1\ldots\int_0^1 & G_\beta(s_1,\ldots,s_{k-1},t_1,\ldots,t_{k-1},N)\cdot H_\beta(s_1,\ldots,s_{k-1},t_1,\ldots,t_{k-1},N)\mathrm{d}t_{k-1}\ldots\mathrm{d}t_1\\
        &=\frac{1}{N^{k-2(k-1)\beta}}\sum_{\substack{1\leq i_1,\ldots,i_{k}\leq N\\\mathrm{distinct}}}\prod_{l=1}^{k-1}\leb\left({B\left(x_{i_l},\frac{s_l}{2N^\beta}\right)}\cap{B\left(x_{i_k},\frac{s_l}{2N^\beta}\right)}\right)\\&+\frac{1}{N^{k-2(k-1)\beta}}\sum_{\substack{1\leq i_1,\ldots,i_{k-1\leq N}\\\mathrm{distinct}}}\sum_{j=1}^{k-1}\prod_{l=1}^{k-1}\leb\left({B\left(x_{i_j},\frac{s_l}{2N^\beta}\right)}\cap{B\left(x_{i_l},\frac{s_l}{2N^\beta}\right)}\right)\\
        &=\frac{1}{N^{k-2(k-1)\beta}}\sum_{\substack{1\leq i_1,\ldots,i_{k}\leq N\\\mathrm{distinct}}}\prod_{l=1}^{k-1}\left\{\frac{s_l}{N^\beta}-\left\Vert x_{i_l}-x_{i_k}\right\Vert\right\}^+\\&+\frac{1}{N^{k-2(k-1)\beta}}\sum_{\substack{1\leq i_1,\ldots,i_{k-1}\leq N\\\mathrm{distinct}}}\sum_{j=1}^{k-1}\frac{s_j}{N^\beta}\prod_{\substack{l=1\\j\neq l}}^{k-1}\left\{\frac{s_l}{N^\beta}-\left\Vert x_j-x_{i_l}\right\Vert\right\}^+\\
        &=\frac{1}{N^{k-(k-1)\beta}}\sum_{\substack{1\leq i_1,\ldots,i_{k}\leq N\\\mathrm{distinct}}}\prod_{l=1}^{k-1}s_l\left\{1-\frac{N^\beta\left\Vert x_{i_l}-x_{i_k}\right\Vert}{s_l}\right\}^+\\&+\sum_{j=1}^{k-1}s_j\frac{1}{N^{k-(k-1)\beta}}\sum_{\substack{1\leq i_1,\ldots,i_{k-1}\leq N\\\mathrm{distinct}}}\prod_{\substack{l=1\\j\neq l}}^{k-1}s_l\left\{1-\frac{N^\beta\left\Vert x_j-x_{i_l}\right\Vert}{s_l}\right\}^+.
    \end{align*} 
    Moreover we can write 
    \[
    R_{k,\beta}(\sigma_1,\ldots,\sigma_{k-1},N)=\frac{1}{N^{k-(k-1)\beta}}\sum_{\substack{1\leq i_1,\ldots,i_k\leq N\\\mathrm{distinct}}}\prod_{l=1}^{k-1}\mathds{1}_{[N^\beta\Vert x_{i_l}-x_{i_k}\Vert,\infty)}(\sigma_l)
    \]
    which gives us
    \begin{align*}
       \int_0^{s_1}\ldots\int_0^{s_{k-1}} & R_{k,\beta}(\sigma_1,\ldots,\sigma_{k-1},N)\mathrm{d}\sigma_{k-1}\ldots\mathrm{d}\sigma_{1}\\
       &=\int_0^{s_1}\ldots\int_0^{s_{k-1}}\frac{1}{N^{k-(k-1)\beta}}\sum_{\substack{1\leq i_1,\ldots,i_k\leq N\\\mathrm{distinct}}}\prod_{l=1}^{k-1}\mathds{1}_{[N^\beta\Vert x_{i_l}-x_{i_k}\Vert,\infty)}(\sigma_l)\mathrm{d}\sigma_{k-1}\ldots\mathrm{d}\sigma_{1}\\
       &=\frac{1}{N^{k-(k-1)\beta}}\sum_{\substack{1\leq i_1,\ldots,i_k\leq N\\\mathrm{distinct}}}\prod_{l=1}^{k-1}\int_0^{s_l}\mathds{1}_{[N^\beta\Vert x_{i_l}-x_{i_k}\Vert,\infty)}(\sigma_l)\mathrm{d}\sigma_{l}\\
       &=\frac{1}{N^{k-(k-1)\beta}}\sum_{\substack{1\leq i_1,\ldots,i_k\leq N\\\mathrm{distinct}}}\prod_{l=1}^{k-1}\left\{s_l-{N^\beta\left\Vert x_{i_l}-x_{i_k}\right\Vert}\right\}^+\\
       &=\frac{1}{N^{k-(k-1)\beta}}\sum_{\substack{1\leq i_1,\ldots,i_k\leq N\\\mathrm{distinct}}}\prod_{l=1}^{k-1}s_l\left\{1-\frac{N^\beta\left\Vert x_{i_l}-x_{i_k}\right\Vert}{s_l}\right\}^+.
    \end{align*}
    Notice that if $k=2$, we obtain
    \begin{align*}
    \int_0^1&G_\beta(s_1,t_1,N)\cdot H_\beta(s_1,t_1,N)\mathrm{d}t_1 \\&=\frac{1}{N^{2-\beta}}\sum_{\substack{1\leq i_1,i_{2}\leq N\\\mathrm{distinct}}}s_1\left\{1-\frac{N^\beta\left\Vert x_{i_1}-x_{i_2}\right\Vert}{s_1}\right\}^+ +\frac{s_1}{N^{1-\beta}}
    \end{align*}
    and if $k\geq 3$, 
    \begin{align*}
    \sum_{j=1}^{k-1}&s_j\frac{1}{N^{k-(k-1)\beta}}\sum_{\substack{1\leq i_1,\ldots,i_{k-1\leq N}\\\mathrm{distinct}}}\prod_{\substack{l=1\\j\neq l}}^{k-1}s_l\left\{1-\frac{N^\beta\left\Vert x_j-x_{i_l}\right\Vert}{s_l}\right\}^+\\ & =\frac{1}{N^{1-\beta}}\sum_{j=1}^{k-1}s_j \int_0^{s_1}\ldots\int_0^{s_{k-1}} R_{k,\beta}(\sigma_1,\ldots,\hat{\sigma_j},\ldots\sigma_{k-1},N)\mathrm{d}\sigma_{k-1}\ldots\mathrm{d}\sigma_{1}.
    \end{align*}

    By using Proposition \ref{prop:poiss_corr_eq_def_fun_G_H}, it is immediate that (i) and (ii) are equivalent.
\end{proof}

In \cite{hauke:weak_poiss_corr} the corresponding result was used to show that Poissonian correlations with $\beta_1 > 0$ imply Poissonian pair correlations for all $\beta_2 < \beta_1$. We believe that Proposition~\ref{prop:poiss_corr_eq_def_fun_G_H} can be used to show this for $(k,\beta)$-Poissonian box correlations as well. Unfortunately neither the proof therein transfers to the more general setting nor did we find a different rigorous argument for it. Therefore we can only state a conjecture here.
\begin{conjecture}\label{conj:poiss_k_stronger_impl_weaker}
    Let $0<\beta_1<\beta_2\leq 1$ and let $k\geq 2$. If a sequence $x=\mathseq{x_n}{n\in\N}$ of real numbers has $(k,\beta_2)$-Poissonian box correlations, then it also has $(k,\beta_1)$-Poissonian box correlations.
\end{conjecture}
Let us finish this section by showing the proof of Theorem \ref{thm:poiss_k_corr_and_discrepancy}, which yields an existence theorem for $(k,\beta)$-Poissonian correlations. We introduce the notation
\begin{align}\label{eq:def_F_t_s_N}
\begin{split}    
F_{\beta,k}&(t,s_1,\ldots,s_{k-1},N)\\&:=\card{\mathset{\substack{1\leq i_1,\ldots,i_{k-1}\leq N\\\mathrm{distinct}}}{\left\Vert t-x_{i_1}\right\Vert\leq\frac{s_1}{N^\beta},\ldots,\left\Vert t-x_{i_{k-1}}\right\Vert\leq\frac{s_{k-1}}{N^\beta}}}.
\end{split}
\end{align}

\begin{proof}[Proof of Theorem \ref{thm:poiss_k_corr_and_discrepancy}]
By the assumption on the discrepancy we know that for any $x\in\R^d$ and any $s>0$, the number of points closer to $x$ than $\frac{s}{N^\beta}$ is $N\frac{2s}{N^{\beta}}+o(N^\varepsilon)$. Recall from \eqref{eq:def_F_t_s_N} the definition of $F_{\beta,k}$. Thus,
\begin{align*}
    R_{k,\beta}(s_1,\ldots,s_{k-1},N)&=\frac{1}{N^{k-(k-1)\beta}}\sum_{i=1}^NF_{\beta,k}(x_i,s_1,\ldots,s_{k-1},N)\\
    &=\frac{1}{N^{k-(k-1)\beta}}N\prod_{j=1}^{k-1}\left(N\frac{2s_j}{N^{\beta}}+o(N^\varepsilon)\right)\\
    &=\prod_{j=1}^{k-1}2s_j+o(N^{-(1-\beta)+\varepsilon}).
\end{align*}

Hence it follows that
\[
\lim_{N\to\infty}R_{k,\beta}(s_1,\ldots,s_{k-1},N)=
    \prod_{j=1}^{k-1}2s_j
\]
and the proof is done.
\end{proof}

\section{Gap structure and $(k,1)$-Poissonian correlations}\label{sec:gap_structure}
We have seen that in Theorem~\ref{thm:poiss_k_corr_and_discrepancy} that low-discrepancy sequences, i.e. sequences with order of discrepancy $O(\log(N)/N)$, compare \cite{Nie92}, are typical examples of sequences with $(k,\beta)$-Poissonian correlations for all $0 < \beta < 1$. In this section we will draw our attention to the case of $(k,1)$-Poissonian pair correlations and analyze in detail how the gap structure of a sequence influences the $(k,1)$-pair correlation statistic. In particular, we will see that the finite gap length property constitutes an obstacle for having Poissonian $(k,1)$-pair correlations. Recall from the introduction, Section \ref{sec:intro}, that given an ordered set of numbers $x_1\leq x_2\leq\ldots\leq x_N \in [0,1)$, a gap is one of the distances $\Vert x_2-x_1\Vert,\ldots,\Vert x_{N}-x_{N-1}\Vert,\Vert x_1-x_N\Vert$. Before we prove Theorem~\ref{thm:gap_strucure_poiss_k_corr}, we shortly mention two important examples of (low-discrepancy) sequences which have a finite number of different gap lengths.
\begin{example}
    Let $\alpha\in\R\setminus\Q$, then the sequence
	\[
	x_n=\{n\alpha\}
	\]
	is called a Kronecker sequence. A Kronecker sequence has at most three gaps. The situation is also illustrated in Figure~\ref{fig:kronecker}. This result is known as the famous Three Gap Theorem and has originally been proved in \cite{Sos58}.  
\end{example}
\begin{figure}[!ht]
    \centering
    \includegraphics[scale=0.4]{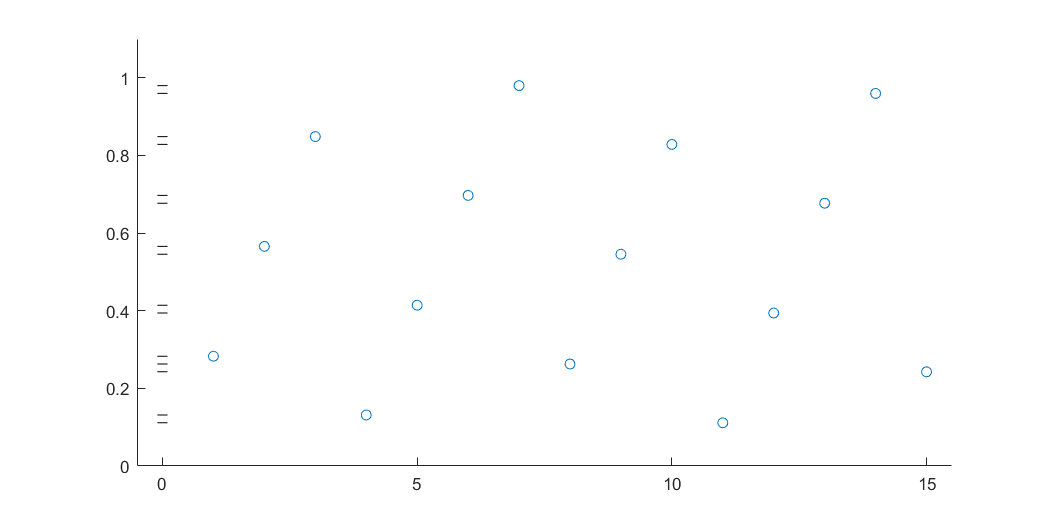}
    \caption{The Kronecker sequence for $\alpha=\frac15\sqrt{2}$. The gap structure can be observed on the left.}
    \label{fig:kronecker}
\end{figure}
\begin{example}
    If $b$ is prime, any number $n\in\N$ can uniquely be represented as $n=a_0+a_1b+a_2b^2+\ldots + a_kb^k$ with $0\leq a_i<b$. Define $g_b(n)=a_0b^{-1}+a_1b^{-2}+\ldots + a_k b^{-(k+1)}$. Then the sequence 
	\[
	x_n=\begin{cases}
	    0&\text{if } n=0\\
        g_b(n)&\text{if } n\geq 1
	\end{cases}
	\]
	is called van der Corput sequence in base $b$. A van der Corput sequence has at most three different gap lengths. If $b=2$, it has at most two gap lengths. This can easily be seen: if $2\leq n<b-1$, then there are exactly two gap lengths. If $n=b^m-1$ for some $m\in\N$, then there is exactly one gap lengths. Adding another element adds two different gap lengths, one between the new element and $0$ and another between the new element and the next elements (these two coincide for $b=2$). The following elements up until $n=b^{m+1}-1$ do not add additional gap lengths.  Often, the element $x_0=0$ is omitted. In this case, the van der Corput sequence can have up to four different gap lengths.
\end{example}
A collection of further examples can e.g. be found in \cite{larcher:som_neg_results_poiss_pair_corr}. To facilitate the proof of Theorem \ref{thm:gap_strucure_poiss_k_corr} we introduce the notation of $\delta$-covers, see e.g. \cite{GPW21}.
\begin{definition}Let $\delta>0$ and $\mathcal{A}\subseteq\mathcal{P}([0,1]^{d})$, a subset of the power set of $[0,1]^d$. A $\delta$-cover for $\mathcal{A}$ is a set $\Gamma\subseteq \mathcal{A}$ such that for every $A\in\mathcal{A}$ there exist $U,V\in \Gamma\cup\{\emptyset\}$ satisfying $U\subseteq A\subseteq V$ and
\[
\leb(V\setminus U)\leq \delta.
\]
\end{definition}

We will specifically use $\delta$-covers for the family
\[
\mathcal{C}(X):=\mathset{[0,x)}{x\in X},
\]
where $X\subseteq \R^d$. Now we have all tools at hand to prove Theorem~\ref{thm:gap_strucure_poiss_k_corr} in a rather compact way.

\begin{proof}[Proof of Theorem~\ref{thm:gap_strucure_poiss_k_corr}]
The sequence $\mathseq{x_n}{n\in\N}$ contains a subsequence $\mathseq{x_{N_i}}{i\in\N}$, such that every $x_1,\ldots,x_{N_i}$ has exactly $Z$ distinct distances between neighboring elements (pigeonhole principle). Let us call these distances $d_1^i,\ldots,d_Z^i$ and order them $d_1^i<d_2^i<\ldots<d_Z^i$. \\[12pt]
We will also without loss of generality assume that the set $x_{1},\ldots,x_{N_i}$ is ordered, meaning that gaps are always between consecutive elements.\\[12pt]
By moving to subsequences if necessary, we may furthermore without loss of generality assume that every $\mathseq{d_j^i}{i\in\N}$ with $1\leq j\leq Z$ falls into one of three categories:
\begin{itemize}
    \item $\lim_{i\to\infty}N_id_j^i=\infty$, which we call large gaps.
    \item $d_j^i=0$ for all $i\in\N$. We call these gaps zero gaps.
    \item There does not exist a subsequence of large gaps and no $d_j^i$ equals zero. In this case, we speak of medium gaps. Note that this implies that there are constants $0<U_j<\infty$ such that $0< N_i d_j^i\leq U_j$ for all $i\in\N$.
\end{itemize}
From now on, we assume that the sequence has $(k,1)$-Poissonian box correlations for some $k>0$, aiming for a contradiction. For every $\vec{n}=(n_1,\ldots,n_{k-1})\in[Z]^{k-1}$, let 
\[
l_{\vec{n}}^i=\card{\mathset{1\leq i_0\leq N_{i_0} \, }{ \, \Vert x_{i_0}-x_{i_0+1}\Vert=d_{n_1}^i,\ldots,\Vert x_{i_0}-x_{i_0+k}\Vert=d_{n_{k-1}}^i}}.
\]
By definition
\begin{equation}\label{eq:sum_vec_n_eq_M_i}
\sum_{\vec{n}\in[Z]^{k-1}}l_{\vec{n}}^i=N_i-k+1
\end{equation}
is true for all $i\in\N$. We will now show that for every $\vec{n}\in[Z]^{k-1}$ without a zero gap $\lim_{i\to\infty}\frac{l_{\vec{n}}^i}{N_i}=0$ holds, i.e. those gaps contribute to the sum in \eqref{eq:sum_vec_n_eq_M_i} in a negligible way. Afterwards we will use this and \eqref{eq:sum_vec_n_eq_M_i} to arrive at a contradiction in the case when $\vec{n}$ represents a sequence of gaps including a zero gaps.\\[12pt]
First, we take a look at the case where there is at least one $j\in\vec{n}$ such that $d_j^i$ is a large gap. Interestingly, for this case the $(k,\beta)$ Poissonian correlations property is irrelevant. Notice that 
\[
l_{\vec{n}}^i\leq\card{\mathset{1\leq i_0\leq N_i-1}{\Vert x_{i_0}-x_{i_0+1}\Vert=d_{j}^i}}\leq\frac{1}{d_j^i}.
\]
It follows that
\[
    \lim_{i\to\infty}\frac{l_{\vec{n}}^i}{N_i}\leq\lim_{i\to\infty}\frac{1}{d_j^iN_i}=0.
\]
Now, let us take a look at the case where all gaps are medium gaps.  We will once again show that $\lim_{i\to\infty}\frac{l_{\vec{n}}^i}{N_i}=0$.
Assume to the contrary that there is some $\delta>0$ such that $ l_{\vec{n}}^i\geq \delta N_i$ for infinitely many $i \in \mathbb{N}$. 
 Since the gaps are medium, there are constants $0<U_j<\infty$ such that $0\leq d_{n_j}^i N_i\leq U_j$ for all $i\in\N$ and $1\leq j<Z$. 
Define $\hat{d}_j^i=\sum_{r=1}^j d_{n_r}^i$ and $\hat{U}_{j}= \sum_{r=1}^j U_{n_r}$, then $0\leq \hat{d}_j^i N_i\leq \hat{U}_j$ for all $i\in\N, 1\leq j<Z$. Take $\Gamma$ to be a $\delta/2^k$-Cover of $\mathcal{C}(\prod_{j=1}^{k-1}[0,\hat{U}_j])$. Then there are $u,t\in\Gamma\cup\{0\}$ with $u<t$, $\leb([0,t)\setminus[0,u))\leq \delta/2^k$ and
\begin{equation}\label{eq:u_low_d_leq_t}
 \frac{u_j}{N_i}< \hat{d}_{j}^i\leq\frac{ t_j}{N_i}   
\end{equation}
for infinitely many $i\in\N$ and all $1\leq j<Z$. By moving to a subsequence, we may assume that this is true for all $i$.  \\[12pt]
For every $i_0\leq N_i$ with 
\[
\Vert x_{i_0}-x_{i_0+1}\Vert=d_{n_1},\ldots,\Vert x_{i_{0}+k-1}-x_{i_0+k}\Vert=d_{n_{k-1}}
\]
we have 
\[
\Vert x_{i_0}-x_{i_0+1}\Vert=\hat{d}_{n_1},\ldots,\Vert x_{i_{0}}-x_{i_0+k}\Vert=\hat{d}_{n_{k-1}}.
\]
By \eqref{eq:u_low_d_leq_t} this implies
$$(x_{i_0},x_{i_0+1},\ldots,x_{i_0+k})\in R_{k,1}(t_{n_{1}},\ldots,t_{n_{k-1}},M_i)$$ and 
$$(x_{i_0},x_{i_0+1},\ldots,x_{i_0+k})\notin R_{k,1}(u_{n_{1}},\ldots,u_{n_{k-1}},M_i).$$ 
Since there are at least $ l_{\vec{n}}^i$ tuples of elements having distances $d_{n_1}^i,\ldots,d_{n_{k-1}}^i$, this now gives us
\begin{align*}
    R_{k,1}(t_{n_{1}},\ldots,t_{n_{k-1}},M_i)
    &\geq R_{k,1}(u_{n_1},\ldots,u_{n_{k-1}},M_i)+l_{\vec{n}}^i.
\end{align*}
As the sequence has $(k,1)$-Poissonian box correlations by assumption, this implies
\begin{align*}
    \prod_{j\in\vec{n}}2t_{j}-\prod_{j\in\vec{n}}2u_{j}
    &=\lim_{i\to\infty} \frac{1}{N_i}R_{k,1}\left(t_{n_{1}},\ldots,t_{n_{k-1}}, N_i\right)-\frac{1}{N_i}R_{k,1}(u_{n_1},\ldots,u_{n_{k-1}}, N_i)\\
    &\geq\limsup_{i\to\infty}\frac{l_{\vec{n}}^i}{N_i}\geq \delta.
\end{align*}
But we chose our points in such a way that 
\[
\prod_{j\in\vec{n}}2t_{j}-\prod_{j\in\vec{n}}2u_{j}\leq 2^{k-1}\frac{\delta}{2^k}<\delta,
\]
which is a contradiction. Therefore,  $\lim_{i\to\infty}\frac{l_{\vec{n}}^i}{N_i}=0$.\\[12pt]
Finally we come to the case where there exists a zero gap. Since the $d_1^i<d_2^i\ldots<d_Z^i$ are ordered by size, $\vec{n}$ represents a zero gap if and only if $1\in\vec{n}$. From what we have shown for large and medium gaps, using \eqref{eq:sum_vec_n_eq_M_i}, we conclude that 
\begin{align*}
 \lim_{i\to\infty} \sum_{\substack{\vec{n}\in[Z]^{k-1}\\1\in\vec{n}}}\frac{l_{\vec{n}}^i}{N_i}=1.
\end{align*}
Therefore,
\begin{align*} 
  \frac{1}{N_i}\card{\mathset{1\leq i_0\leq N_i-1\,}{\,x_{i_0}=x_{i_0+1}}} =\sum_{\substack{\vec{n}\in[Z]^{k-1}\\1\in\vec{n}}}\frac{l_{\vec{n}}^i}{N_i}\stackrel{i\to\infty}{\longrightarrow}1.
\end{align*}
Setting $Q_i:=\card{\mathset{1\leq i_0\leq N_i-1}{x_{i_0}=x_{i_0+1}}}$ it follows that $\lim_{i\to\infty}Q_i=\infty$. Thus
\begin{align*}
    0&=\lim_{i\to\infty}\frac{1}{N_i}R_{k,1}(0,\ldots,0,N_i) = \infty.
\end{align*}
This contradiction finishes the proof.
\end{proof}
Notice that, since $(k,\beta)$-Poissonian correlations imply $(k,\beta)$-Poissonian box correlations for all $0 < \beta \leq 1$, any sequence satisfying the conditions of Theorem \ref{thm:gap_strucure_poiss_k_corr} also does not have $(k,1)$-Poissonian correlations.

\section{Future research} \label{sec:future}
There are several question left open by this paper which may be studied in the future. The first question was stated as Conjecture~\ref{conj:poiss_k_stronger_impl_weaker}, i.e. if $(k,\beta_2)$-Poissonian box correlations imply $(k,\beta_1)$-Poissonian box correlations for $\beta_2<\beta_1$? The second question is whether or not there is a difference between the concepts of $(k,\beta)$-Poissonian correlations and $(k,\beta)$-Poissonian box correlations. If those two are not equivalent, it would be interesting to study examples of sequences which satisfy the latter, but not the former definition. Furthermore one might wonder if there are specific conditions on the parameters $k$ and $\beta$ such the result from Theorem \ref{thm:gap_strucure_poiss_k_corr} still holds or how counter examples can be constructed. Finally it would be interesting to find meaningful generalizations to higher dimensions, similar to \cite{NP07}, \cite{hinrichs:mult_dim_poiss_pair} or \cite{chaubey:pair_corr_real_val_vector}, and to investigate what properties these have.

\bibliographystyle{alpha}
\bibdata{references}
\bibliography{references}
\end{document}